\documentclass[10pt,twoside,a4paper]{article}
\usepackage{amsfonts,amssymb,amsmath,amscd,latexsym,makeidx,theorem}
\author{Luca Martinazzi}
\title{Classification of solutions to the higher order Liouville's equation on $\R{2m}$}

\newtheorem{trm}{Theorem}
\newtheorem{prop}[trm]{Proposition}
\newtheorem{cor}[trm]{Corollary}
\newtheorem{lemma}[trm]{Lemma}

\newcommand{\R}[1]{\mathbb{R}^{#1}}

\newcommand{\ve}{\varepsilon}
\newcommand{\bs}{\backslash}

\newenvironment{proof}{\noindent\emph{Proof.}}{\hfill$\square$\medskip}

\newenvironment{rmk}{\medskip\noindent\emph{Remark.}}{\hfill$\bullet$\medskip}

\DeclareMathOperator{\diver}{div}
\DeclareMathOperator{\loc}{loc}

\DeclareMathOperator*{\Intm}{\int\!\!\!\!\!\! \rule[2.6pt]{6.5pt}{.4pt}}
\DeclareMathOperator{\intm}{\int\!\!\!\!\!\!\: --}
\DeclareMathOperator{\Ric}{Ric}
\DeclareMathOperator{\vol}{vol}

\begin{document}
\maketitle

\begin{abstract}
We classify the solutions to the equation $(-\Delta)^m u=(2m-1)!e^{2mu}$ on $\R{2m}$ giving rise to a metric $g=e^{2u}g_{\R{2m}}$ with finite total $Q$-curvature in terms of analytic and geometric properties. The analytic conditions involve the growth rate of $u$ and the asymptotic behaviour of $\Delta u(x)$ as $|x|\to\infty$. As a consequence we give a geometric characterization in terms of the scalar curvature of the metric $e^{2u}g_{\R{2m}}$ at infinity, and we observe that the pull-back of this metric to $S^{2m}$ via the stereographic projection can be extended to a smooth Riemannian metric if and only if it is round.
\end{abstract}

\section{Introduction and statement of the main theorems}

The study of the Paneitz operators has moved into the center of conformal geometry in the last decades, in part with regard to the problem of prescribing the $Q$-curvature. 
Given a $4$-dimensional Riemannian manifold $(M,g)$, the $Q$-curvature $Q^4_g$ and the Paneitz operator $P^4_g$ have been introduced by Branson-Oersted \cite{BO} and Paneitz \cite{pan}:
\begin{eqnarray*}
Q^4_g&:=& -\frac{1}{6}\big(\Delta_g R_g - R^2_g +3 |\Ric_g|^2\big)\\
P^4_g(f)&:=&\Delta^2_g f +\diver\Big(\frac{2}{3} R_g g -2\Ric_g\Big)df,\quad \forall f\in C^{\infty}(M),
\end{eqnarray*}
where $R_g$ and $\Ric_g$ denote the scalar and Ricci curvatures of $g$.
Higher order $Q$-curvatures $Q^n$ and Paneitz operators $P^n$ have been introduced in \cite{bra} and \cite{GJMS}. Their interest lies in their covariant nature: considering in dimension $2m$ the conformal metric $g_u:=e^{2u}g$, we have
\begin{equation}\label{gauss}
P^{2m}_{g_u}=e^{-2m u}P^{2m}_g, \quad P^{2m}_g u +Q^{2m}_g=Q_{g_u}^{2m} e^{2mu},
\end{equation}
see for instance \cite{cha} Chapter 4. The last identity is a generalized version of Gau\ss's identity: in dimension $2$
$$-\Delta_{g} u +K_g=K_{g_u}e^{2u},$$
where $K_g$ is the Gaussian curvature, and $\Delta_g$ is the Laplace-Beltrami operator with the analysts' sign. Indeed, in dimension 2 we have $P^2_g=-\Delta_g$ and $Q^2_g=K_g$. Moreover $\Delta_{g_u}=e^{-2u}\Delta_g$.
Another interesting fact is that the total $Q$-curvature is a global conformal invariant: if $M$ is closed and $2m$-dimensional,
$$\int_M Q_{g_u}^{2m}\mathrm{dvol}_{g_{u}}=\int_M Q_g^{2m} \mathrm{dvol}_g.$$
Further evidence of the geometric relevance of the $Q$-curvatures is given by the Gauss-Bonnet-Chern's theorem \cite{che}: on a locally conformally flat closed mani\-fold of dimension $2m$, since $Q^{2m}_g$ is a multiple of the Pfaffian plus a divergence term (see \cite{BGP}), we have
$$\int_M Q^{2m}_g \mathrm{dvol}_g =[(2m-2)!!]^2\vol(S^{2m-1})\chi(M),$$
where $\chi(M)$ is the Euler-Poincar\'e characteristic of $M$.

\medskip

Here we are interested in the special case when $M$ is $\R{2m}$ with the Euclidean metric $g_{\R{2m}}$. In this case we simply have $P^{2m}_{g_{\R{2m}}}=(-\Delta)^m$ and $Q^{2m}_{g_{\R{2m}}}\equiv 0$.
We consider solutions to the equation

\begin{equation}\label{eq0}
(-\Delta)^m u=(2m-1)!e^{2mu} \quad \textrm{on }\R{2m},
\end{equation}
satisfying $\int_{\R{2m}}e^{2mu}dx<\infty.$
From the above remarks and \eqref{gauss} in particular, it follows that \eqref{eq0} has the following geometric meaning: if $u$ solves \eqref{eq0}, then the conformal metric $g:=e^{2u}g_{\R{2m}}$ has $Q$-curvature $Q_g^{2m}\equiv (2m-1)!$. As we shall see, every solution to \eqref{eq0} with $e^{2mu}\in L^1_{\loc}(\R{2m})$ is smooth (Corollary \ref{smooth}).

Given such a solution $u$, define the auxiliary function
\begin{equation}\label{eqv}
v(x):= \frac{(2m-1)!}{\gamma_m}\int_{\R{2m}}\log\bigg(\frac{|y|}{|x-y|}\bigg) e^{2m u(y)}dy,
\end{equation}
where $\gamma_m$ is defined by the following property: $(-\Delta)^m\big(\frac{1}{\gamma_m}\log\frac{1}{|x|}\big)=\delta_0$ in $\R{2m}$, see Proposition \ref{fund} below. Then
$(-\Delta)^m v=(2m-1)!e^{2mu}.$ We prove

\begin{trm}\label{clas1} Let $u$ be a solution of \eqref{eq0} with
\begin{equation}\label{area}
\alpha:=\frac{1}{|S^{2m}|}\int_{\R{2m}} e^{2mu(x)}dx<+\infty.
\end{equation}
Then
\begin{equation}\label{repr}
u(x)=v(x)+p(x),
\end{equation}
where $p$ is a polynomial of even degree at most $2m-2$, $v$ is as in \eqref{eqv} and
\begin{eqnarray*}
\sup_{x\in \R{2m}} p(x)&<& +\infty,\\
\lim_{|x|\to\infty}\Delta^j v(x)&=&0,\quad j=1,\ldots, m-1,\\
v(x)&=&-2\alpha\log|x|+o(\log|x|),\textrm{ as }|x|\to+\infty.
\end{eqnarray*}
\end{trm}

It is well known that the function
\begin{equation}\label{standard}
u(x):=\log\frac{2\lambda}{1+\lambda^2|x-x_0|^2}
\end{equation}
solves \eqref{eq0} and \eqref{area} with $\alpha=1$ for any $\lambda>0$, $x_0\in\R{2m}$. We call the functions of the form \eqref{standard} \emph{standard solutions}. They all arise as pull-back under the stereographic projection of metrics on $S^{2m}$ which are round, i.e. conformally diffeomorphic to the standard metric. A. Chang and P. Yang \cite{CY} proved that the round metrics are the only metrics on $S^{2m}$ having $Q$-curvature identically equal to $(2m-1)!$.

In the next theorem we give conditions under which an entire solution of Liouville's equation satisfying \eqref{area} is necessarily a standard solution.

\begin{trm}\label{clas2} Let $u$ be a solution of \eqref{clas2} satisfying \eqref{area}. Then the following are equivalent:
\begin{itemize}
\item[(i)] $u$ is a standard solution,
\item[(ii)] $\lim_{|x|\to\infty}\Delta u(x)=0$
\item[(ii')] $\lim_{|x|\to\infty}\Delta^j u(x)=0$ for $j=1,\ldots,m-1$,
\item[(iii)] $u(x)=o(|x|^2)$ as $|x|\to \infty$,
\item[(iv)] $\deg p=0$, where $p$ is the polynomial in \eqref{repr}.
\item[(v)] $\liminf_{|x|\to +\infty}R_{g_u}> -\infty$, where $g_u=e^{2u}g_{\R{2m}}$.
\item[(vi)] $\pi^*g_u$ can be extended to a Riemannian metric on $S^{2m}$, where $\pi:S^{2m}\to\R{2m}$ is the stereographic projection.
\end{itemize}
Moreover, if $u$ is not a standard solution, there exist $1\leq j\leq m-1$ and a constant $a<0$ such that
\begin{equation}\label{deltaa}
\Delta^j u(x)\to a\quad \textrm{as }|x|\to+\infty.
\end{equation}
\end{trm}

The $2$-dimensional case ($m=1$) of Theorem \ref{clas2} was treated by W. Chen and C. Lin \cite{CL}, who proved that \emph{every} solution with finite total Gaussian curvature is a standard one. The $4$-dimensional case was treated by C-S. Lin \cite{lin}, with a classification of $u$ in terms of its growth, or of the behaviour of $\Delta u$ at $\infty$. The classification of C-S. Lin in terms of $\Delta u$ was used by F. Robert and M. Struwe \cite{RS} to study the blow-up behaviour of sequences of solutions $u_k$ to
$$
\left\{
\begin{array}{ll}
\Delta^2 u_k=\lambda u_k e^{32 \pi^2 u_k^2} & \textrm{in }\Omega\subset\R{4}\\
u_k=\frac{\partial u_k}{\partial n}=0&\textrm{on } \partial \Omega,
\end{array}
\right.
$$
and by A. Malchiodi \cite{mal} to show a compactness criterion for sequences of solutions $u_k$ to the equation
$$P^4_g u_k +Q^4_k=h_ke^{4u_k},\quad h_k \textrm{ constant}$$
on a closed $4$-manifold.
The same criterion could be used in higher dimension in the proof of an analogous compactness result. This was observed by C. B. Ndiaye \cite{ndi}, who then used a different technique to show compactness. We will discuss this in a forthcoming paper.

In higher dimension ($m>2$), J. Wei and X. Xu \cite{WX} treated a special case of Theorem \ref{clas2}: if $u(x)=o(|x|^2)$ at infinity, then $u$ is always a standard solution. This result is not sufficient to prove compactness.
Moreover, the proof appears to be overly simplified. For instance, in their Lemma 2.2 the argument for showing that $u\leq C$ is not conclusive, and in the crucial Lemma 2.4 they simply refer to \cite{lin} for details.
This latter lemma corresponds to Lemma \ref{Deltav} here and it is the main regularity result, as it implies that $u\leq C$, hence that the volume of the metric $e^{2u}g_{\R{2m}}$ cannot concentrate in small balls. Its generalization is a major issue, because Lin's analysis is focused on the function $\Delta u$, and it  makes use of the Harnack's inequality and of the fact that $\Delta(u-v)\equiv C$. In the general case, Harnack's inequality does not work and there are no uniform bounds for $\Delta^{(m-2)}(u-v)$ (while it is still true that $\Delta^{(m-1)}(u-v)\equiv C$). To overcome this difficulties, we spend a few pages in the following section to study polyharmonic functions. As a reward we obtain a Liouville-type theorem for polyharmonic functions (Theorem \ref{trmliou2}) which allows us to make the proof of \cite{lin} more direct and transparent.

\medskip

The characterization in terms of the scalar curvature at infinity is new and quite interesting, as it shows that non-standard solutions have a geometry essentially different from standard solutions, and it also shows that the $Q$-curvature and the scalar curvature are independent of each other in dimension $4$ and higher. On the other hand, since in dimension $2$ we have $2Q_g=R_g$, (v) is consistent with the result of \cite{CL}.

The characterization in (vi) implies the result of A. Chang and P. Yang \cite{CY} described above, which here follows from the general case.

\medskip

The paper is organized as follows. In Section \ref{afew} we collect some relevant results about polyharmonic functions which will be needed later. Section \ref{main} contains the proof of Theorems \ref{clas1} and \ref{clas2}; at the end of the paper we give examples to show that the hypothesis of Theorem \ref{clas2} are sharp in terms of the growth at infinity and of the degree of $p$. Recently J. C. Wei and D. Ye \cite{WY} proved that already in dimension $4$ there is a great abundance of non-radially symmetric solutions.

\medskip

In the following, the letter $C$ denotes a generic constant, which may change from line to line and even within the same line.

\subsubsection*{Aknowledgments}

I wish to thank my advisor, Prof. M. Struwe for stimulating discussions and for introducing me to this very interesting subject. I also thank my friend D. Saccavino for referring me to the result of Gorin, which we use in Lemma \ref{polinf}.

\section{A few remarks on polyharmonic functions}\label{afew}

We briefly recall some properties of polyharmonic functions, which will be used in the sequel. For the standard elliptic estimates for the Laplace operator, we refer to \cite{GT} or \cite{GM}. The next lemma can be considered a generalized mean value inequality. We give the short proof for the convenience of the reader, and because identity \eqref{+} will be used in the next section.

\begin{lemma}[Pizzetti \cite{Piz}] Let $\Delta^m h=0$ in $B_R(x_0)\subset \R{n}$, for some $m,n$ positive integers. Then
\begin{equation}\label{pizzetti}
\Intm_{B_R(x_0)}h(z)dz =\sum_{i=0}^{m-1}c_iR^{2i}\Delta^i h(x_0),
\end{equation}
where
\begin{equation}\label{coeff}
c_0=1,\quad c_i=\frac{n}{n+2i}\frac{(n-2)!!}{(2i)!!(2i+n-2)!!},  \quad i\geq 1.
\end{equation}
\end{lemma}

\begin{proof} We can translate and assume that $x_0=0$. We first prove by induction on $m$ that there are constants $b_0^{(m)},\ldots, b_{m-1}^{(m)}$ such that
\begin{equation}\label{pizze}
\Intm_{\partial B_r}h(z)dS=\sum_{i=0}^{m-1}b_i^{(m)}r^{2i}\Delta^i h(0), \quad 0<r<R, \; B_r:=B_r(0).
\end{equation}
For $m=1$ this reduces to the mean value theorem for harmonic functions. Assume now that the assertion has been proved up to $m-1$, and that $\Delta^m h=0$. Let $G_r$ be the Green function of $\Delta^m$ in $B_r$:
\begin{equation}\label{cond}
\Delta^m G_r=\delta_0\textrm{ in } B_r,\quad G_r=\Delta G_r=\ldots=\Delta^{m-1}G_r=0 \textrm{ on } \partial B_r.
\end{equation}
For simplicity, let us only consider the case $n=2m$. 
Then $G_r(x)=G_1\big(\frac{x}{r}\big)$,
$$G_1(x)=\alpha_0 \log|x|+\alpha_1|x|^2+\ldots+ \alpha_{m-1} |x|^{2m-2},$$
where the constants can be computed inductively starting with $\alpha_0$ up to $\alpha_{m-1}$ in order to satisfy \eqref{cond}. Notice that $G_1$ is radial.
Integrating by parts
\begin{eqnarray}
0&=&\int_{B_r}G_r\Delta^m hdx\nonumber\\
&=&h(0)-\sum_{i=0}^{m-1}\int_{\partial B_r} \frac{\partial \Delta^{m-1-i} G_r}{\partial n}\Delta^{i}h dS\label{+}\\
&=&h(0)-\sum_{i=0}^{m-1}\Intm_{\partial B_r} a_i r^{2i}\Delta^i hdS,\nonumber
\end{eqnarray}
where each $a_i$ depends only on $n$ and $m$. For each term on the right-hand side with $i\geq 1$, we can use the inductive hypothesis
$$r^{2i}\Intm_{\partial B_r}\Delta^i h dS=r^{2i}\sum_{j=0}^{m-i-1}b_j^{(m-1)}r^{2j}\Delta^{j+i}h(0), \quad 0\leq i\leq m-1,$$
and substituting we obtain \eqref{pizze}. To conclude the induction it is enough to multiply \eqref{pizze} by $r^{n-1}$, integrate with respect to $r$ from $0$ to $R$ and divide by $\frac{R^n}{n}$.

To compute the $c_i$'s, we test with the functions $h(x)=r^{2i}:=|x|^{2i}$, $i\geq 1$ (for the case $i=0$ use the function $h(x)\equiv 1$). Since $\Delta r^{2i}=2i(2i+n-2)r^{2i-2}$, we have that $\Delta^k h(0)=0$ for $k\neq i$ and $\Delta^i h(0)=\frac{(2i)!!(2i+n-2)!!}{(n-2)!!}$. Hence Pizzetti's formula reduces to
$$c_i R^{2i} \frac{(2i)!!(2i+n-2)!!}{(n-2)!!}=\Intm_{B_R}r^{2i}dx=\frac{n}{n+2i}R^{2i},$$
whence \eqref{coeff}.
\end{proof}

\begin{rmk}\label{rmkpiz} From \eqref{+}, moreover, for an arbitrary $C^{2m}$-function $u$ it follows that
\begin{equation}\label{pizgen}
\Intm_{B_R(x_0)}u(z)dz=\sum_{i=0}^{m-1}c_iR^{2i}\Delta^i u(x_0)+c_{m}R^{2m}\Delta^{m} u(\xi),
\end{equation}
for some $\xi\in B_R(x_0)$.
\end{rmk}

\begin{prop}\label{c2m} Let $\Delta^{m}h=0$ in $B_{4}\subset\R{n}$. For every $0\leq \alpha<1$, $p\in [1,\infty)$ and $k\geq 0$ there are  constants $C(k,p), C(k,\alpha)$ independent of $h$ such that
\begin{eqnarray*}
\|h\|_{W^{k,p}(B_1)}&\leq &C(k,p) \|h\|_{L^1(B_4)}\\
\|h\|_{C^{k,\alpha}(B_1)}&\leq& C(k,\alpha)  \|h\|_{L^1(B_4)}.
\end{eqnarray*}
\end{prop}

The proof of Proposition \ref{c2m} is given in the appendix.
As a consequence of Proposition \ref{c2m} and Pizzetti's formula we have the following Liouville-type theorem, compare \cite{ARS}.

\begin{trm}\label{trmliou} Consider $h:\R{n}\to \R{}$ with $\Delta^m h=0$ and $h(x)\leq C(1+|x|^\ell)$, for some $\ell\geq 2m-2$. Then $h(x)$ is a polynomial of degree at most $\ell$. 
\end{trm}

\begin{proof}
Thanks to Proposition \ref{c2m}, we have for any $x\in\R{n}$
\begin{equation}\label{lioupr}
|D^{\ell+1}h(x)|\leq\frac{C}{R^{\ell+1}}\Intm_{B_R(x)}|h(y)|dy=-\frac{C}{R^{\ell+1}}\Intm_{B_R(x)}h(y)dy+O(R^{-1}),\quad\textrm{as }R\to\infty.
\end{equation}
On the other hand, Pizzetti's formula implies that
$$\Intm_{B_R(x)}h(y)dy=\sum_{i=0}^{m-1}c_i R^{2i}\Delta^i h(x)=O(R^{2m-2}),$$
and letting $R\to\infty$, we obtain $D^{\ell+1}h=0$.
\end{proof}

A variant of the above theorem, which will be used later is the following

\begin{trm}\label{trmliou2} Consider $h:\R{n}\to \R{}$ with $\Delta^m h=0$ and $h(x)\leq u-v$, where $e^{pu}\in L^1(\R{n})$ for some $p>0$,  $v\in L^1_{\loc}(\R{n})$ and $-v(x)\leq C(\log(1+|x|)+1)$. Then $h$ is a polynomial of degree at most $2m-2$. 
\end{trm}

\begin{proof} The only thing to change in the proof of Theorem \ref{trmliou}, is the estimate of the term $\frac{2C}{R^{2m-1}}\intm_{B_R(x)} h^+dy$, corresponding to the $O(R^{-1})$ in \eqref{lioupr}.
We have
\begin{eqnarray*}
\Intm_{B_R(x)}h^+dy&\leq& \Intm_{B_R(x)} u^+dy+C\Intm_{B_R(x)}\log(1+|y|)dy+C\\
&\leq&\frac{1}{p}\Intm_{B_R(x)}e^{pu}dy+C\log R+C,
\end{eqnarray*}
and all terms go to $0$ when divided by $R^{2m-1}$ and for $R\to\infty$.
\end{proof}

The following estimate has been obtained by Br\'ezis and Merle \cite{BM} in dimension $2$ and by C.S. Lin \cite{lin} and J. Wei \cite{wei} in dimension $4$. Notice that the constant $\gamma_m$, defined by the relation
$$(-\Delta)^m\bigg(\frac{1}{\gamma_m}\log\frac{1}{|x|}\bigg)=\delta_0,\quad \textrm{in }\R{2m}$$
(see Proposition \ref{fund} in the appendix), plays an important role.

\begin{trm}\label{a2m} Let $f\in L^1(B_R(x_0))$ and let $v$ solve
$$\left\{
\begin{array}{ll}
(-\Delta)^m v=f &\textrm{in } B_R(x_0)\subset \R{2m},\\
v=\Delta v=\ldots=\Delta^{m-1}v=0& \textrm{on }\partial B_R(x_0).
\end{array}
\right.
$$
Then, for any $p\in\Big(0,\frac{\gamma_m}{\|f\|_{L^1(B_R(x_0))}}\Big)$, we have $e^{2m p|v|}\in L^1(B_R(x_0))$ and
$$\int_{B_R(x_0)}e^{2m p |v|}dx\leq C(p)R^{2m},$$
where $\gamma_m$ is given by \eqref{gammam}.
\end{trm}

\begin{proof} We can assume $x_0=0$ and, up to rescaling, that $\|f\|_{L^1(B_R)}=1$. Define
$$w(x):=\frac{1}{\gamma_m}\int_{B_R}\log\frac{2R}{|x-y|}|f(y)|dy,\quad x\in\R{2m}.$$
Extend $f$ to be zero outside $B_R(x_0)$; then
$$(-\Delta)^m w=|f|\quad\textrm{in } \R{2m}.$$
We claim that $w\geq |v|$ in $B_R$. Indeed by \eqref{eqdelta3} and from $|x-y|\leq 2R$ for $x,y\in B_R$, we immediately see that
$$(-\Delta)^j w\geq 0,\quad j=0,1,2,\ldots$$
In particular the function $z:=w-v$ satisfies
$$
\left\{
\begin{array}{ll}
(-\Delta)^m z\geq 0 & \textrm{in }B_R\\
(-\Delta)^j z\geq 0 & \textrm{on }\partial B_R\textrm{ for } 0\leq j\leq m-1.
\end{array}
\right.
$$
By Proposition \ref{propmax}, $(-\Delta)^jz\geq 0$ in $B_R$, $0\leq j\leq m-1$ and the case $j=0$ corresponds $w\geq v$. Working also with $-v$ we complete the proof of our claim.

Now it suffices to show that for $p\in (0,\gamma_m)$ we have $\|e^{2mpw}\|_{L^1(B_R)}\leq C(p)R^{2m}.$ By Jensen's inequality we have

\begin{eqnarray*}
\int_{B_R}e^{2mpw}dx&=&\int_{B_R}e^{\frac{2mp}{\gamma_m}\int_{B_R}\log\frac{2R}{|x-y|}|f(y)|dy}dx\\
&\leq&\int_{B_R}\int_{B_R}|f(y)|e^{\frac{2mp}{\gamma_m}\log\frac{2R}{|x-y|}}dydx\\
&=&\int_{B_R}|f(y)|\bigg(\int_{B_R}\bigg(\frac{2R}{|x-y|}\bigg)^\frac{2mp}{\gamma_m}dx\bigg)dy
\end{eqnarray*}
On the other hand
\begin{eqnarray*}
\int_{B_R}\bigg(\frac{2R}{|x-y|}\bigg)^\frac{2mp}{\gamma_m}dx&\leq&
\int_{B_R}\bigg(\frac{2R}{|x|}\bigg)^\frac{2mp}{\gamma_m}dx\\
&=&\omega_{2m}\int_0^R r^{2m-1-\frac{2mp}{\gamma_m}}(2R)^\frac{2mp}{\gamma_m}dr\\
&=&\omega_{2m}\frac{\gamma_m}{2m\gamma_m-2mp}R^{2m}2^\frac{2mp}{\gamma_m}.
\end{eqnarray*}
We then conclude
$$\int_{B_R}e^{2mpw}dx\leq \frac{C(m)}{\gamma_m-p}R^{2m}.$$
\end{proof}

\begin{cor}\label{smooth} Every solution $u$ to \eqref{eq0} with $e^{2mu}\in L^1_{\loc}(\R{2m})$ is smooth. 
\end{cor}

\begin{proof} Given $B_4(x_0)\subset\R{2m}$, write $(2m-1)!e^{2mu}\big|_{B_4(x_0)}=f_1+f_2$ with
$$\|f_1\|_{L^1(B_4(x_0))}<\gamma_m, \quad f_2\in L^\infty(B_4(x_0)),$$
and $u=u_1+u_2+u_3$, with
$$
\left\{
\begin{array}{ll}
(-\Delta)^m u_i=f_i &\textrm{in } B_4(x_0)\\
u_i=\Delta u_i=\ldots=\Delta^{m-1} u_i=0& \textrm{on }\partial B_4(x_0)
\end{array}
\right.
$$
for $i=1,2$, and $\Delta^m u_3=0$. Then, by Theorem \ref{a2m}, $e^{2mu_1}\in L^p(B_4(x_0))$ for some $p>1$, while, by standard elliptic estimates $u_2\in L^\infty (B_4(x_0))$ and $u_3$ is smooth, hence $u_3\in L^\infty(B_3(x_0))$. Then $e^{2mu}\in L^p(B_3(x_0))$. Write now $u\big|_{B_3(x_0)}=v_1+v_2$, where
$$
\left\{
\begin{array}{ll}
(-\Delta)^m v_1=(2m-1)!e^{2mu} &\textrm{in } B_3(x_0)\\
v_1=\Delta v_1=\ldots=\Delta^{m-1} v_1=0& \textrm{on }\partial B_3(x_0)
\end{array}
\right.
$$
and $\Delta^m v_2=0$. Then, by $L^p$-estimates and Sobolev's embedding theorem, $v_1\in W^{2m,p}(B_3(x_0))\hookrightarrow C^{0,\alpha}(B_3(x_0))$ for some $0<\alpha<1$, while $v_2$ is smooth. Then $u\in C^{0,\alpha}(B_2(x_0))$ and with the same procedure of writing $u$ as the sum of a polyharmonic (hence smooth) function plus a function with vanishing Navier boundary condition, we can bootstrap and use Schauder's estimate to prove that $u\in C^\infty(B_1(x_0))$. 
\end{proof}

\section{Proof of Theorems \ref{clas1} and \ref{clas2}}\label{main}

The proof of Theorems \ref{clas1} and \ref{clas2} will be divided into several lemmas. It consists of a careful study of the functions $v$, defined in \eqref{eqv}, and $u-v$. In what follows the generic constant $C$ may depend also on $u$.

\begin{rmk}
In general $v\neq u$, even if $u$ is a standard solution. To see that, rescale $u$ by a factor $r>0$ as follows:
$$\widetilde u(x):=u(rx)+\log r.$$ Then $\widetilde u$ is again a solution, with the same energy. On the other hand the corresponding $\widetilde v$ satisfies
\begin{eqnarray}
\widetilde v(x)&=&\frac{(2m-1)!}{\gamma_m}\int_{\R{2m}}\log\bigg(\frac{|y|}{|x-y|}\bigg) e^{2m u(ry)}r^{2m}dy\nonumber\\
&=&\frac{(2m-1)!}{\gamma_m}\int_{\R{2m}}\log\bigg(\frac{|y'|}{|rx-y'|}\bigg) e^{2m u(y')}dy'=v(rx).\label{scalev}
\end{eqnarray}
That shows that after rescaling, $u-v$ changes by a contant.
\end{rmk}

\begin{lemma}\label{lemmabeta} Let $u$ be a solution of \eqref{eq0}, \eqref{area}.
Then, for $|x|\geq 4$,
\begin{equation}\label{eqlog}
v(x)\geq -2\alpha \log|x|+C.
\end{equation}
\end{lemma}

\begin{proof} The proof is similar to that in dimension $4$, compare \cite{lin}. Fix  $x$ with $|x|\geq 4$, and decompose $\R{2m}=A_1\cup A_2\cup B_2$, where $B_2=B_2(0)$ and
$$A_1:=B_{|x|/2}(x), \quad A_2:= \R{2m}\backslash (A_1\cup B_2).$$
For $y\in A_1$ we have
$$|y|\geq |x|-|x-y|\geq \frac{|x|}{2}\geq |x-y|,\quad \log \frac{|y|}{|x-y|}\geq 0,$$
hence
\begin{equation}\label{A1}
\int_{A_1}\log\frac{|y|}{|x-y|}e^{2mu(y)}dy\geq 0.
\end{equation}
For $y\in A_2$, since $|x|,|y|\geq 2$, we have
$$|x-y|\leq |x|+|y|\leq |x||y|,\quad \log\frac{|y|}{|x-y|}\geq \log\frac{1}{|x|},$$
hence
\begin{equation}\label{A2}
\int_{A_2}\log\frac{|y|}{|x-y|}e^{2mu(y)}dy\geq -\log|x|\int_{A_2}e^{2mu(y)}dy.
\end{equation}
For $y\in B_2$, $\log|x-y|\leq \log |x|+C$ and, since $u$ is smooth, we find
\begin{eqnarray}
\int_{B_2}\log\frac{|y|}{|x-y|}e^{2mu(y)}dy&\geq&\int_{B_2}\log|y|e^{2mu(y)}dy-\log|x|\int_{B_2}e^{2mu}dy\nonumber\\
&&-C\int_{B_2}e^{2mu}dy\nonumber\\
&\geq& -\log|x|\int_{B_2}e^{2mu}dy+C.\label{B2}
\end{eqnarray}
Putting together \eqref{A1}, \eqref{A2} and \eqref{B2} and observing that $\log\frac{1}{|x|}<0$, we conclude that
\begin{eqnarray*}
v(x)&\geq& \frac{(2m-1)!}{\gamma_m}\int_{A_2\cup B_2}\log\bigg(\frac{|y|}{|x-y|}\bigg) e^{2m u(y)}dy\\
&\geq&-\frac{(2m-1)!}{\gamma_m}\log|x|  \int_{A_2\cup B_2} e^{2mu}dy+C\\
&\geq&-\frac{(2m-1)!|S^{2m}|}{\gamma_m}\alpha \log|x| +C.
\end{eqnarray*}
Finally, observing that $(2m-2)!!=2^{m-1}(m-1)!$, we infer
$$\frac{(2m-1)!|S^{2m}|}{\gamma_m}=\frac{(2m-1)! 2(2\pi)^m (2m-2)!!}{(2m-1)!!2^{3m-2}[(m-1)!]^2\pi^m}=2.$$
\end{proof}

\begin{lemma}\label{Deltapol} Let $u$ be a solution of \eqref{eq0} and \eqref{area}, with $m\geq 2$.
Then $u=v+p$, where $p$ is a polynomial of degree at most $2m-2$. Moreover
\begin{eqnarray*}
\Delta^{j} u(x)&=&  \Delta^{j}v(x) + p_j\\
&=&(-1)^{j} \frac{2^{2j}(j-1)!(m-1)!}{(m-j-1)!|S^{2m}|}\int_{\R{2m}}\frac{e^{2mu(y)}}{|x-y|^{2j}}dy +p_j,
\end{eqnarray*}
where $p_j$ is a polynomial of degree at most $2(m-1-j)$.
\end{lemma}

\begin{proof} Let $p:=u-v$. Then $\Delta^m p=0$. By Lemma \ref{lemmabeta} we have
$$p(x)\leq u(x)+2\alpha \log|x| +C,$$
and Theorem \ref{trmliou2} implies that $p$ is a polynomial of degree at most $2m-2$. To compute $\Delta^{j} v$, one can use \eqref{eqdelta3} and the definition of $\gamma_m$.
\end{proof}

\begin{lemma}\label{polinf} Let $p$ be the polynomial of Lemma \ref{Deltapol}. Then
$$\sup_{x\in\R{2m}}p(x)<+\infty.$$
In particular $\deg p$ is even.
\end{lemma}

\begin{proof} Define
$$f(r):=\sup_{\partial B_r} p.$$
If $\sup_{\R{2m}} p=+\infty$, there exists $s>0$ such that
\begin{equation}\label{fr}
\lim_{r\to+\infty}\frac{f(r)}{r^s}=+\infty,
\end{equation}
see \cite[Theorem 3.1]{gor}.\footnote{The statement of Theorem 3.1 in \cite{gor} is about $\mu(r):=\inf_{\partial B_r} |p|$, but the proof works in our case too.}
Moreover $|\nabla p(x)|\leq C|x|^{2m-3}$ hence, also taking into account Lemma \ref{lemmabeta}, there is $R>0$ such that for every $r\geq R$, we can find $x_r$ with $|x_r|=r$ such that
$$u(y)=v(y)+p(y)\geq r^s\quad \textrm{for } |y-x_r|\leq\frac{1}{r^{2m-3}}.$$
Then, using Fubini's theorem,
\begin{eqnarray*}
\int_{\R{2m}}e^{2mu}dx&\geq& \int_R^{+\infty}\int_{\partial B_r(0)\cap B_{r^{3-2m}}(x_r)}e^{2mr^s}d\sigma dr\\
&\geq& C\int_R^{+\infty}\frac{\exp(2mr^s)}{r^{(2m-3)(2m-1)}}dr=+\infty,
\end{eqnarray*}
contradicting the hypothesis $e^{2mu}\in L^1(\R{2m})$.
\end{proof}

The following lemma will be used in the proof of Lemma \ref{Deltav}.

\begin{lemma}\label{segnogreen}
Let $G=G(|x|)$ be the Green's function for $\Delta^m$ in $B_1\subset\R{n}$ for $n$, $m$ given positive integers. Then there are constants $c_i$ depending on $m$ and $n$ such that for $|x|=1$, and $0\leq i\leq m-1,$
$$(-1)^i\frac{\partial \Delta^{m-1-i}G(x)}{\partial r}=c_i>0.$$
\end{lemma}

\begin{proof} Since $G=G(|x|)$, we only need to show that $c_i>0$. Fix $i$ and let $h$ solve
$$\left\{
\begin{array}{ll}
\Delta^m h=0 &\textrm{in } B_1\\
(-\Delta)^i h =-1 & \textrm{on }\partial B_1\\
(-\Delta)^j h =0 & \textrm{on }\partial B_1 \textrm{ for }0\leq j\leq m-1,\;j\neq i.
\end{array}
\right.
$$
By Proposition \ref{propmax}, $h(0)<0$, hence \eqref{+} implies
$$0<-h(0)=(-1)^i\int_{\partial B_1}\frac{\partial \Delta^{m-1-i}G}{\partial r}dS= c_i\omega_n.$$
\end{proof}

\begin{lemma}\label{Deltav} Let $v:\R{2m}\to\R{}$ be defined as in \eqref{eqv}. Then
\begin{equation}\label{dv}
\lim_{|x|\to\infty}\Delta^{m-j} v(x)=0, \quad j=1,\ldots,m-1
\end{equation}
and for any $\ve>0$ there is $R>0$ such that for $|x|>R$
\begin{equation}\label{veps}
v(x)\leq (-2\alpha +\varepsilon)\log|x|.
\end{equation}
\end{lemma}

\begin{proof} We proceed by steps.

\medskip

\noindent\emph{Step 1.} For any $\ve>0$ there is $R>0$ such that for $|x|\geq R$
\begin{equation}\label{211}
v(x)\leq \Big(-2\alpha+\frac{\ve}{2}\Big)\log|x| -\frac{(2m-1)!}{\gamma_m}\int_{B_\tau(x)}\log |x-y|e^{2m u(y)}dy,
\end{equation}
where $\tau\in (0,1)$ will be fixed later. The simple proof of \eqref{211} is very similar to the proof of Lemma \ref{lemmabeta} (see \cite[Pag. 213]{lin}), and it is omitted. Notice that the second term on the right-hand side may be very large.
Together with Fubini's theorem, \eqref{211} implies
\begin{eqnarray}
\int_{\R{2m}\bs B_R(0)}v^+dx&\leq& C\int_{\R{2m}}\int_{\R{2m}}\chi_{|x-y|\leq \tau}\log\frac{1}{|x-y|} e^{2mu(y)}dydx\nonumber\\
&=& C \int_{\R{2m}}e^{2mu(y)}\int_{B_\tau(y)}\log\frac{1}{|x-y|}dxdy \nonumber\\
&\leq &C\int_{R^{2m}}e^{2mu(y)}dy\leq C.\label{v+}
\end{eqnarray}

\medskip

\noindent\emph{Step 2.} From now on, $x$ will be a point in $\R{2m}$ with $|x|>R$, where $R$ is as in Step 1. Fix $p>1$ such that $p(2m-2)<2m$, and $p'=\frac{p}{p-1}$. By Theorem \ref{a2m}, there is $\delta>0$ such that if
\begin{equation}\label{b4x}
\int_{B_4(x)}e^{2mu}dy<\delta,
\end{equation}
then
\begin{equation}\label{zb4}
\int_{B_4(x)}e^{2mp' |z|}dy\leq C,
\end{equation}
with $C$ independent of $x$, where $z$ solves

$$\left\{
\begin{array}{ll}
(-\Delta)^m z= (2m-1)!e^{2mu} & \textrm{in }B_{4}(x)\\
\Delta^j z=0 &\textrm{on }\partial B_{4}(x) \textrm{ for } 0\leq j\leq m-1.
\end{array}
\right.$$
We now choose $R>0$ such that \eqref{b4x} is satisfied whenever $|x|\geq R$, and claim that for such $x$,
\begin{equation}\label{2mp}
\int_{B_\tau(x)}e^{2mp'u}dy\leq C\int_{B_\tau(x)} e^{2mp'|z|}dy\leq C\ve.
\end{equation}
We now observe that for any $\sigma>0$,
\begin{equation}\label{leb}
\int_{\R{2m}\backslash B_\sigma(x)}\frac{e^{2mu(y)}}{|x-y|^{2j}}dy\to 0\quad\textrm{as }|x|\to\infty
\end{equation}
by dominated convergence; by H\"older's inequality and \eqref{2mp}, if $\sigma$ is small enough,
$$\int_{B_\sigma(x)}\frac{e^{2mu}}{|x-y|^{2j}}dy\leq\bigg(\int_{B_\sigma(x)} e^{2mp'u}dy\bigg)^{\frac{1}{p'}}\bigg(\int_{B_\sigma(x)}\frac{1}{|x-y|^{2jp}}dy\bigg)^\frac{1}{p}\leq C\ve^\frac{1}{p'}. $$
Therefore
$$(-\Delta)^{j}v(x)=C\int_{\R{2m}}\frac{e^{2mu}}{|x-y|^{2j}}dy\to 0,\quad \textrm{as } |x|\to\infty.$$
Finally \eqref{veps} follows from \eqref{211}, \eqref{2mp} and H\"older's inequality.

\medskip

\noindent\emph{Step 3.} It remains to prove \eqref{2mp}. Set $h:=v-z$, so that
$$\left\{
\begin{array}{ll}
\Delta^m h= 0 & \textrm{in }B_{4}(x)\\
\Delta^j h=\Delta^j v &\textrm{on }\partial B_{4}(x) \textrm{ for } 0\leq j\leq m-1,
\end{array}
\right.$$
Integrating $(-\Delta)^m v=(2m-1)!e^{2mu}$ and then integrating by parts we get
$$(-1)^m\int_{\partial B_\rho(x)}\frac{\partial}{\partial r}(\Delta^{m-1} v)dS=(2m-1)!\int_{B_\rho(x)}e^{2mu}dy.$$
Dividing by $\omega_{2m}\rho^{2m-1}$, integrating on $[0,R]$ and using Fubini's, we find
\begin{multline*}
\int_0^R\Intm_{\partial B_\rho(x)}\frac{\partial}{\partial r}(\Delta^{m-1} v)d\sigma d\rho=\int_0^R\Intm_{\partial B_1(x)}\frac{\partial}{\partial r}(\Delta^{m-1} v(\rho,\theta))d\theta d\rho\\
=\Intm_{\partial B_1(x)}\int_{0}^R\frac{\partial}{\partial r}(\Delta^{m-1} v(\rho,\theta))d\rho d\theta
=\Intm_{\partial B_R(x)} \Delta^{m-1} v d\sigma -\Delta^{m-1} v(x).
\end{multline*}
Similarly
\begin{multline*}
\int_0^R\frac{1}{\rho^{2m-1}}\int_{B_\rho(x)}e^{2mu(y)}dyd\rho=
\int_0^R\frac{1}{\rho^{2m-1}}\int_{B_R(x)}e^{2mu(y)}\chi_{|x-y|\leq \rho}dyd\rho\\
=\int_{B_R(x)}e^{2mu(y)}\int_{|x-y|}^R\frac{1}{\rho^{2m-1}}d\rho dy\\
=\frac{1}{(2m-2)}\int_{B_R(x)}\bigg[\frac{1}{|x-y|^{2m-2}}-\frac{1}{R^{2m-2}}\bigg]e^{2mu(y)}dy.
\end{multline*}
Hence, multiplying above by $\frac{(2m-1)!}{\omega_{2m}}$ and setting $C_{m-1}:=\frac{(2m-1)!}{(2m-2)\omega_{2m}}$,
\begin{eqnarray*}
\Intm_{\partial B_R} (-\Delta)^{m-1} v d\sigma &=&(-\Delta)^{m-1}v(x)\\
&&- C_{m-1}\int_{B_R(x)}\bigg[\frac{1}{|x-y|^{2m-2}}-\frac{1}{R^{2m-2}}\bigg]e^{2mu(y)}dy\\
&=&C_{m-1} \bigg[\int_{|x-y|\geq R}\frac{e^{2m u(y)}}{|x-y|^{2m-2}}dy+\int_{B_R(x)}\frac{e^{2mu(y)}}{R^{2m-2}}dy\bigg]
\end{eqnarray*}
which implies at once, setting $R=4$,
\begin{equation}\label{Du4}
\Intm_{\partial B_4(x)}(-\Delta)^{m-1}vdS\leq C,
\end{equation}
with $C$ independent of $x$.
Similarly, one can show that
\begin{equation}\label{Du5}
\Intm_{\partial B_4(x)}(-\Delta)^{i}vdS\leq C,\quad 1\leq i\leq m-1.
\end{equation}
By Lemma \ref{segnogreen} and by \eqref{+} rescaled and translated to $B_4(x)$ and with the function $-\Delta h$ instead of $h$, $m-1$ instead of $m$, we obtain
\begin{eqnarray}
-\Delta h(x)&=&-\sum_{i=0}^{m-2}\int_{\partial B_4(x)}\frac{\partial \Delta^{m-1-i}G}{\partial n}\Delta^i(\Delta h)dS\label{deltah}\\
&=&\sum_{i=1}^{m-1}\int_{\partial B_4(x)}c_{i-1}(-\Delta)^ihdS\leq C,\nonumber
\end{eqnarray}
where $G$ is the Green function for $\Delta^{m-1}$ on $B_4(x)$:
$$\Delta^{m-1} G=\delta_x,\quad \Delta^iG=0,\textrm{ on }\partial B_4(x), \textrm{ for }0\leq i\leq m-2.$$
On the other hand, since the $c_i>0$, there is some $\tau>0$ such that the following holds: if $\xi\in B_{2\tau}(x)$ and $G_\xi$ is the Green's function defined by
$$\Delta^{m-1} G_\xi=\delta_\xi,\quad \Delta^iG_\xi=0,\textrm{ on }\partial B_4(x), \textrm{ for }0\leq i\leq m-2,$$
then also
$$0\leq (-1)^i\frac{\partial \Delta^{m-1-i}G_\xi(\eta)}{\partial r}\leq C, \quad \textrm{for }\eta\in\partial B_4(x),\;r:=\frac{\eta-x}{4}.$$
Therefore, as in \eqref{deltah}, we infer
\begin{equation}\label{hx}
-\Delta h\leq C\quad \textrm{on }B_{2\tau}(x),
\end{equation}
for some $\tau\in (0,2)$.
\medskip

On the other hand, thanks to \eqref{v+} and \eqref{zb4},
$$\int_{B_4(x)}h^+dy\leq \int_{B_4(x)} (v^++|z|)dy\leq C.$$
By elliptic estimates,
$$\sup_{B_\tau(x)}h \leq \Intm_{B_4(x)}h^+dy+ C\sup_{B_{2\tau(x)}}(-\Delta h)\leq C,$$
$C$ independent of $x$, as usual.
Since the polynomial $p$ is bounded from above, we infer
$$u\leq h+p+|z|\leq C +|z|,$$ and \eqref{2mp} follows at once.
\end{proof}

\begin{cor} Any solution $u$ of \eqref{eq0}, \eqref{area} is bounded from above.
\end{cor}

\begin{proof} Indeed $u$ is continuous, $u=v+p$, and
$$\lim_{|x|\to\infty}v(x)=-\infty,\quad \sup_{|x|\in \R{2m}}p(x)<+\infty,$$
by Lemma \ref{polinf}.
\end{proof}

\begin{lemma}\label{lemmao} Assume that $|u(x)|=o(|x|^2)$ as $|x|\to\infty$. Then $u=v+C$. Furthermore, for any $\ve>0$ there exists $R>0$ such that
\begin{equation}\label{226}
-2\alpha \log|x|-C\leq u(x)\leq (-2\alpha+\ve)\log|x|,
\end{equation}
for $|x|\geq R$.
\end{lemma}

\begin{proof} Since $v(x)= -2\alpha\log|x|+o(\log|x|)$ at $\infty$, if $\deg p\geq 2$, we have that $u(x)=v(x)+p(x)$ cannot be $o(|x|^2)$. Hence, knowing that $\deg p$ is even, we get $u=v+C$ for some constant $C$. Then \eqref{226} follows at once from Lemma \ref{lemmabeta} and Lemma \ref{Deltav}.
\end{proof}

\begin{lemma}\label{scalar}  Set $g_u=e^{2u}g_{\R{2m}}$. If $u$ is a standard solution, then
$$R_{g_u}\equiv 2m(2m-1).$$
If $u$ is not a standard solution, then
\begin{equation}\label{Rinf}
\liminf_{|x|\to+\infty} R_{g_u}(x) =-\infty.
\end{equation} 
\end{lemma}

\begin{proof}  Assume that $u$ is a standard solution and set
\begin{equation}\label{g1}
u_\lambda(x):=\log\frac{2\lambda}{1+\lambda^2 |x|^2},\quad g_\lambda:=e^{2u_\lambda}g_{\R{2m}}.
\end{equation}
Then, up to translation, $u=u_\lambda$ for some $\lambda>0$. Since
$g_1=(\pi^{-1})^*g_{S^{2m}},$ where $\pi$ is the stereographic projection,
we have $R_{g_1}\equiv 2m(2m-1)$. Then consider the diffeomorphism of $\R{2m}$ defined by $\varphi_\lambda(x):=\lambda x$. Then $g_\lambda=\varphi_\lambda^* g_1$, hence $R_{g_\lambda}=R_{g_1}\circ \varphi_\lambda\equiv 2m(2m-1)$.

Assume now that $u=v+p$ is not a standard solution.
Since $g_{\R{2m}}$ is flat, the formula for the conformal change of scalar curvature, in the case $m>1$, reduces to
\begin{equation}\label{cambiaR}
R_{g_u}=-2(2m-1)e^{-2u}\Big(\Delta u+(m-1)|\nabla u|^2\Big),
\end{equation}
see for instance \cite{SY} pag 184.
Then differentiating the expression \eqref{eqv} for $v$ and using that $u\leq C$, we find that $|\nabla v(x)|\to 0$ as $|x|\to \infty$. We have already seen that $\Delta v(x)\to 0$ as $|x|\to \infty$; since $\deg p\geq 2$ implies
$$\deg \Delta p<\deg |\nabla p|^2,$$
we then have 
$$\limsup_{|x|\to\infty} \Big(\Delta u+(m-1)|\nabla u|^2\Big)=\limsup_{|x|\to\infty} \Big(\Delta p+(m-1)|\nabla p|^2\Big)=+\infty.$$
Observing that $e^{-2u}\geq \frac{1}{C}>0$, $u$ being bounded from above, we easily obtain \eqref{Rinf}.
\end{proof}

\medskip

\noindent\emph{Proof of Theorem \ref{clas1}.} Put together Lemmas \ref{lemmabeta}, \ref{Deltapol}, \ref{polinf} and \ref{Deltav}.
\hfill $\square$

\medskip

\noindent\emph{Proof of Theorem \ref{clas2}.}  (i) $\Rightarrow$ (iii) is obvious, while (iii) $\Rightarrow$ (i) follows from the argument of \cite{WX}.

\medskip

\noindent (iii) $\Leftrightarrow$ (iv) follows from Theorem \ref{clas1}.

\medskip

\noindent (iv) $\Rightarrow$ (ii') $\Rightarrow$ (ii). Assume that $\deg p=0$. Then by Theorem \ref{clas1},
$$\lim_{|x|\to\infty} \Delta^j u(x)=\lim_{|x|\to|\infty}\Delta^j p(x)=0, \quad 1\leq j\leq m-1.$$

\medskip

\noindent (ii) $\Rightarrow$ (iv). By Theorem \ref{clas1}, $\sup_{\R{2m}} p<\infty$ and
$$\lim_{|x|\to\infty}\Delta p(x)=\lim_{|x|\to\infty}\Delta u=0,$$
hence $\Delta p\equiv 0$ and, by Liouville's theorem, $p$ is constant.
\medskip

\medskip

\noindent (i) $\Leftrightarrow$ (v) follows from Lemma \ref{scalar}.

\medskip

\noindent (i) $\Rightarrow$ (vi) Given a conformal diffeomorphism $\varphi$ of $\R{2m}$, $\widetilde\varphi:=\pi^{-1}\circ \varphi\circ\pi$ is a conformal diffeomorphism of $S^{2m}$. Any metric of the form $g_u=e^{2u}g_{\R{2m}}$, with $u$ standard solution of \eqref{eq0}, can be easily written as $\varphi^* g_1$, for some conformal diffeomorphism $\varphi$ of $\R{2m}, $where $g_1$ is as in \eqref{g1}. Then
$$\pi^* g_u=\pi^*\varphi^*g_1=(\varphi\circ\pi)^* g_1=(\pi\circ\widetilde \varphi)^*g_1=\widetilde\varphi^*\pi^* g_1=\widetilde \varphi^*g_{S^2},$$
and clearly $\widetilde \varphi^*g_{S^2}$ is a smooth Riemannian metric on $S^{2m}$.

\medskip

\noindent (vi) $\Rightarrow$ (i). Assume $u$ is non-standard. Then $u=v+p$, $\deg p\geq 2$. Considering that $\sup_{\R{2m}}p<+\infty$, we infer that $p$ goes to $-\infty$ at least quadratically in some directions. Let $S=(0,\ldots,0,1)\in S^{2m}$ be the South Pole, and 
$$\pi:S^{2m}\backslash\{S\} \to \R{2m},\quad \pi(\xi):=\frac{(\xi_1,\ldots,\xi_{2m})}{1+\xi_{2m+1}}$$
be the stereographic projection from $S$.
Then
$$(\pi^{-1})^* g_{S^{2m}}=\rho_0 g_{\R{2m}},\quad \rho_0(x):=\frac{4}{(1+|x|^2)^2}, $$
and
$$\pi^* g_u=\rho_1 g_{S^{2m}},\quad \rho_1:=\frac{e^{2u}}{\rho_0}\circ\pi\in C^\infty(S^{2m}\backslash \{S\}).$$
Since $e^{2u(x)}\to 0$ more rapidly than $|x|^{-4}$ in some directions, we have
$$\liminf_{\xi\to S}\rho_1(\xi)=\liminf_{|x|\to \infty}\frac{e^{2u(x)}}{\rho_0(x)}=0,$$
hence $\rho_1 g_{S^{2m}}$ does not extend to a Riemannian metric on $S^{2m}$.

\medskip

To prove \eqref{deltaa}, let $j$ be the largest integer such that $\Delta^j p\neq 0$. Then $\Delta^{j+1} p\equiv 0$ and from Theorem \ref{trmliou2} we infer that $\deg p\leq 2j$. In fact $\deg p=2j$ and $\Delta^j p\equiv C_0\neq 0$.
From Pizzetti's formula \eqref{pizze}, we have
$$2m\sum_{i=0}^j b_iR^{2i}\Delta^i p(0)=\Intm_{\partial B_R}2m p dS$$
Exponentiating and using Jensen's inequality and Lemma \ref{lemmabeta}, we infer
$$\exp\Big(2m\sum_{i=0}^j b_iR^{2i}\Delta^i p(0)\Big)\leq\Intm_{\partial B_R}e^{2mp}dS\leq C R^{4m\alpha}\Intm_{\partial B_R}e^{2mu}d S,$$
for $R\geq 4$. Therefore
$$\varphi(R):= R^{-4m\alpha+2m-1}\exp\Big(2m\sum_{i=0}^j b_iR^{2i}\Delta^i p(0)\Big)\in L^1([4,+\infty)),$$
and this is not possible if $C_0=\Delta^j p > 0$, hence $C_0<0$.

\hfill $\square$

\section{Examples}

Following an argument of \cite{CC}, we now see that solutions of the kind $v+p$ actually exist, even among radially symmetric functions, with $\deg p=2m-2$, and with $\deg p=2$. For simplicity, we only treat the case when $m$ is even; if $m$ is odd, the proof is similar. We need the following lemma.

\begin{lemma}\label{lapl} Let $u(r)$ be a smooth radially symmetric function on $\R{n}$, $n\geq 1$. Then for $m\geq 0$ we have
\begin{equation}\label{eqlapl}
\Delta^m u(0)=\frac{n}{c_m (n+2m)(2m)!}u^{(2m)}(0),
\end{equation}
where the $c_i$'s are the constants in Pizzetti's formula, and $u^{(2m)}:=\frac{\partial^{2m}u}{\partial r^{2m}}$. In particular $\Delta^mu(0)$ has the sign of $u^{(2m)}(0)$.
\end{lemma}

\begin{proof} We first prove that
\begin{equation}\label{eqlapl2}
c_m\Delta^m u(0)=\frac{1}{R^{2m}}\Intm_{B_R(0)}\frac{r^{2m}}{(2m)!}u^{(2m)}(0)dx.
\end{equation}
Then, observing that
\begin{equation}\label{intR}
\Intm_{B_R(0)}\frac{r^{2m}}{(2m)!}dx=\frac{nR^{2m}}{(n+2m)(2m)!},
\end{equation}
\eqref{eqlapl} follows at once.
We prove \eqref{eqlapl2} by induction. The case $m=0$ reduces to $u(0)=u(0)$. Let us now assume that \eqref{eqlapl2} has been proven for $i=0,\ldots,m-1$ and let us prove it for $m$. Since $u$ is smooth, we have $u^{(i)}(0)=0$ for any odd $i$, hence Taylor's formula reduces to
$$u(r)=\sum_{i=0}^m \frac{r^{2i}}{(2i)!}u^{(2i)}(0)+o(r^{2m+1}).$$
We now divide by $R^{2m}$ in \eqref{pizgen}, take the limit as $R\to 0$ and, observing that $\Delta^{m+1}u(\xi)$ remains bounded as $R\to 0$, we find
\begin{eqnarray*}
\lim_{R\to 0}\frac{\intm_{B_R}\Big(u-\sum_{i=0}^{m-1}c_iR^{2i}\Delta^iu(0)\Big)dx}{R^{2m}}&=&c_m\Delta^mu(0).
\end{eqnarray*}
Substituting Taylor's formula and using the inductive hypothesis, we see that most of the terms on the left-hand side cancel out (before taking the limit) and we are left with
$$\lim_{R\to 0}\frac{1}{R^{2m}}\Intm_{B_R}\bigg(\frac{r^{2m}u^{(2m)}(0)}{(2m)!}+o(r^{2m+1})\bigg)dx=c_m\Delta^m u(0).$$
Finally, to deduce \eqref{eqlapl2}, observe that, $\frac{1}{R^{2m}}\intm_{B_R(0)}o(r^{2m+1})dx\to 0$ as $R\to 0$,
while $\frac{1}{R^{2m}}\intm_{B_R}\frac{r^{2m}u^{(2m)}(0)}{(2m)!}dx$ does not depend on $R$ thanks to \eqref{intR}.
\end{proof}

\begin{prop}\label{esempio} For every $m\geq 2$ \emph{even}, there exists a radially symmetric function $u$ solving \eqref{eq0}, \eqref{area} with $u(x)=-C|x|^{2m-2}+O(|x|^{2m-4})$.
\end{prop}

\begin{proof} Set $w_0=\log\frac{2}{1+r^2}$. Then $\Delta^m w_0=(2m-1)!e^{2mw_0}$. Define $u=u(r)$ to be the unique solution to the following ODE
$$
\left\{
\begin{array}{ll}
\Delta^m u=(2m-1)!e^{2mu}&\\
u(0)=\log 2&\\
u^{(2j+1)}(0)=0&j=0,\ldots,m-1\rule{0cm}{.4cm}\\
u^{(2j)}(0)=\alpha_j\leq w_0^{(2j)}(0)&j=1,\ldots,m-2\rule{0cm}{.4cm}\\
u^{(2m-2)}(0)=\alpha_{m-1}<w_0^{(2m-2)}(0)&\rule{0cm}{.4cm}
\end{array}
\right.
$$
where the $\alpha_j$'s are fixed.
We shall first see that $w_0\geq u$. Set $g:=w_0-u$. Then $g(r)> 0$ for $r>0$ small enough, hence also $\Delta^m g> 0$ for small $r>0$. From Lemma \ref{lapl} we get
\begin{equation}\label{deltai}
\Delta^{j}g(0)\geq 0,\quad j=1,\ldots, m-2;\qquad \Delta^{m-1}g(0)>0.
\end{equation}
We can prove inductively that $\Delta^{m-j} g\geq 0$, $j=0,\ldots,m-1$ as long as $g(r)>0$. Indeed
\begin{equation}\label{intdelta}
\int_{B_R(0)}\Delta^j g dx=\int_{\partial B_R(0)} \frac{\partial \Delta^{j-1}g}{\partial r}d\sigma,
\end{equation}
hence, as long as $g(r)> 0$, we have $\frac{\partial \Delta^{j-1}g}{\partial r}> 0$, in particular $\frac{\partial g}{\partial r}> 0$, hence $g(r)>0$ for all $r>0$ for which it is defined. 
From \eqref{deltai} and \eqref{intdelta} we inductively infer
$$\Delta^{m-j}g(r)\geq Cr^{2j-2},$$
and, since $\Delta w_0(r)\to 0$ as $r\to\infty$, there is $r_0>0$ such that
$$\Delta u\leq -Cr^{2m-4},\quad \textrm{for }r\geq r_0,$$
integrating which, we find
\begin{equation}\label{asym}
u(r)\leq -Cr^{2m-2}\quad \textrm{for } r\geq r_0.
\end{equation}
To estimate $u$ from below, we use the function
$$w_1(r)=\log 2- C_1r^2 -\ldots -C_{m-1}r^{2m-2},$$
where the constants $C_i$ are chosen so that
$$\Delta^j u(0)\geq \Delta^j w_1(0).$$
Then we can proceed as above to prove that $u-w_1\geq 0$. Hence the solution exists for all times and, thanks to \eqref{asym} and Theorem \ref{clas1}, it has the asymptotic behaviour
$$u(r)=-Cr^{2m-2}+O(r^{2m-4}).$$ 
\end{proof}

\begin{rmk} Observe the abundance of solutions: we can choose the $(m-1)$-tuple of initial data $(\alpha_1,\ldots,\alpha_{m-1})$ in a set containing an open subset of $\R{m-1}$.
\end{rmk}

In the next example we show a radially symmetric solution in $\R{2m}$, $m\geq 4$ even, of the form $u=v+p$, with $\deg p=2$, thus showing that the hypothesis $u(x)=o(|x|^2)$ as $|x|\to \infty$ in Theorem \ref{clas2} is sharp.

\begin{prop} Let $w_0(r):=\log\frac{2}{1+r^2}$ and let $u=u(r)$ ($r=|x|$, $x\in\R{2m}$ and $m$ even) solve the following ODE:
$$
\left\{
\begin{array}{ll}
\Delta^m u=(2m-1)!e^{2mu}&\\
u(0)=\log 2&\\
u^{(2j+1)}(0)=0& j=0,\ldots, m-1\rule{0cm}{.4cm}\\
u^{(2j)}(0)=w_0^{(2j)}(0)&j=2,3,\ldots,m-1\rule{0cm}{.4cm}\\
u''(0)=w_0''(0)-1.\rule{0cm}{.4cm}
\end{array}
\right.
$$
Then $u(r)$ is defined for all $r\geq 0$ and $u(r)= -C r^2+o(r^2)$ as $r\to+\infty$.
\end{prop}

\begin{proof}
As in the proof of Proposition \ref{esempio}, we can show that $g:= w_0-u\geq 0$ and $u(r)\leq -Cr^2$. To control $u$ from below, we use the function $w_1(r)=w_0(r)-r^2$, so that redefining $g:=u-w_1$, we have
$$g''(0)=1,\quad g^{(j)}(0)=0,\quad j=0,1,3,4,\ldots,2m-1.$$
and we can prove that $g\geq 0$ as before. Hence $u(r)$ exists for all $r\geq 0$, it is non-standard and $u(r)= -C r^2+O(r^4)$ at $\infty$, as $w_1$ bounds it from below.
\end{proof}

\begin{rmk} Using \eqref{cambiaR}, we can easily compute that in the above examples
$$\lim_{|x|\to \infty}R_g(x)\to -\infty,$$
where $g=e^{2u}g_{\R{2m}}$.
\end{rmk}

\section*{Appendix}

We prove here a few results used above.

\begin{lemma}\label{wkp} Assume that $u:B_4\rightarrow\R{}$ satisfies
\begin{eqnarray*}
\|\Delta u\|_{W^{k,p}(B_4)}&\leq& C\\
\|u\|_{L^1(B_4)}&\leq& C,
\end{eqnarray*}
for some $p\in (1,\infty)$. Then
$$\|u\|_{W^{k+2,p}(B_1)}\leq C.$$
\end{lemma}

\begin{proof}
By Fubini's theorem we can choose $r>0$ with $2\leq r\leq 4$ such that
$$\|u\|_{L^1(\partial B_r)}\leq C \|u\|_{L^1(B_4)}.$$
Let's now write $u=u_1+u_2$, where
$$\left\{
\begin{array}{ll}
\Delta u_1=0& \textrm{in }B_r\\
u_1=u&\textrm{on }\partial B_r
\end{array}
\right.
\qquad
\left\{
\begin{array}{ll}
\Delta u_2=\Delta u& \textrm{in }B_r\\
u_2=0&\textrm{on }\partial B_r
\end{array}
\right.
$$
By standard $L^p$-estimates we have $\|u_2\|_{W^{k+2,p}(B_r)}\leq C \|\Delta u\|_{W^{k,p}(B_r)}$. From the representation formula of Poisson
$$u_1(x)=\int_{\partial B_r} u_1(y)\Gamma(x-y)dS(y),$$
we obtain $\|u_1\|_{C^k(B_1)}\leq C_k\|u_1\|_{L^1(\partial B_r)}$ for every $k\geq 0$ .
\end{proof}

\medskip

\noindent\emph{Proof of Proposition \ref{c2m}.} Let $\|h\|_{L^1(B_{4})}\leq C$, and let us assume $n>2$. We proceed by steps.

\medskip

\noindent\emph{Step 1.} We show by induction on $j$ that
\begin{equation}\label{inddelta}
\|\Delta^{m-j} h\|_{L^\infty(B_{2})}\leq C.
\end{equation}
The step $j=0$ is obvious, as $\Delta^m h\equiv 0$. Let us prove the step $j\geq 1$. Let
$$G_{2r}(x):=\frac{1}{(2-n)\omega_n}\bigg(\frac{1}{|x|^{n-2}}-\frac{1}{(2r)^{n-2}}\bigg)$$
be the Green function for the Laplace operator on $B_{2r}$ with singularity at $0$.
Then
$$ \Delta^{m-j}h(0)=\Intm_{\partial B_{2r}} \Delta^{m-j}hdx+\int_{B_{2r}} G_{2r} \Delta^{m-j+1}h dx.$$
By inductive hypothesis and the scaling property of $G_{2r}$, the last term is bounded by $C r^2$, hence
$$\Delta^{m-j}h(0)\leq\Intm_{\partial B_{2r}} \Delta^{m-j}h dx +Cr^2,$$
and integrating with respect to $r$ on $[1/2,1]$, we obtain 
\begin{equation}\label{eqm-j}
\Delta^{m-j}h(0)\leq \Intm_{B_{2}} \Delta^{m-j}h dx +C.
\end{equation}
To estimate $\intm_{B_{2}} \Delta^{m-j}h dx$, we use Pizzetti's formula for $h$ at $x\in B_2$,
$$c_{m-j}\Delta^{m-j}h(x)=-\sum_{i=0}^{m-j-1}c_i\Delta^i h(x)\underbrace{-\sum_{i=m-j+1}^{m}c_i\Delta^i h(x)+\Intm_{B_1(x)}hdy}_{\leq C} $$
by the inductive hypothesis again, and the $L^1$-bound on $h$ and get
\begin{equation}\label{deltaR0}
c_{m-j}\Delta^{m-j}h(x) \leq  -\sum_{i=0}^{m-j-1} c_{i} \Delta^{i} h(x)+ C.
\end{equation}
Averaging in \eqref{deltaR0} over $B_2$ and using \eqref{eqm-j}, we find
\begin{equation*}
c_{m-j}\Delta^{m-j}h(0)\leq -\sum_{i=0}^{m-j-1}\bigg(c_i \Intm_{B_2}\Delta^i h(x)dx\bigg)+C.
\end{equation*}
and its scaled version
\begin{equation}\label{deltar}
c_{m-j}\Delta^{m-j}h(0)\leq -\sum_{i=0}^{m-j-1}\bigg(c_ir^{2(i-m+j)} \Intm_{B_{2r}}\Delta^i h(x)dx\bigg)+Cr^{2(j-m)}.
\end{equation}
Consider now a non-negative function $\varphi\in C^\infty_c((1,2))$, with $\int_1^2\varphi(r) dr=1$. From \eqref{deltar}, we find
$$
c_{m-j}\Delta^{m-j}h(0)\leq -\sum_{i=0}^{m-j-1}c_i\int_1^{2} \bigg(r^{2(i-m+j)} \Intm_{B_{2r}}\Delta^{i} h(x)dx\;\varphi(r)\bigg)dr+C.
$$
Each term in the sum on the right-hand side can be written as
\begin{eqnarray*}
&&\bigg|C\int_1^{2} r^{2(i-m+j)-n} \int_{\partial B_{2r}}\frac{\partial\Delta^{i-1} h}{\partial\nu}dS\varphi(r)dr\bigg|\\
&\leq&C\bigg|\int_{B_{2}\bs B_1} r^{2(i-m+j)-n} \frac{\partial\Delta^{i-1} h(x)}{\partial\nu}\varphi(|x|)dx\bigg|\\
&=&C\int_{B_{2}\bs B_1}|h(x)|\bigg|  \frac{\partial}{\partial\nu}\Delta^{i-1}\big( r^{2(i-m+j)-n}\varphi(|x|)\big)\bigg|dx\\
&\leq& C \Intm_{B_{2}}|h(x)|dx.
\end{eqnarray*}
Working with $-h$ and observing the local character of the above estimates, we obtain \eqref{inddelta}.

\medskip

\noindent\emph{Step 2.} Fix $\ell\geq m$. We can prove inductively that
$$\|\Delta^{\ell-j}h\|_{W^{2j,p}(B_2)}\leq C(p).$$
The step $j=0$ is obvious, as $\Delta^\ell h\equiv 0$. For the inductive step, we see that by Lemma \ref{wkp} applied to $\Delta^{\ell-j}h$ (and a simple covering argument to fix the radii), we have
$$\|\Delta^{\ell-j}h\|_{W^{2j,p}(B_1)}\leq C\|\Delta (\Delta^{\ell-j}h)\|_{W^{2j-2,p}(B_2)}+C\underbrace{\|\Delta^{\ell-j}h\|_{L^1(B_2)}}_{\leq C \textrm{ by Step 1}}\leq C, $$
for every $1 < p< \infty$, and the usual covering argument extends the estimate to $B_2$. Therefore $\|h\|_{W^{2\ell,p}(B_1)}\leq C(p,\ell)$, and we conclude applying Sobolev's theorem.
\hfill$\square$

\begin{prop}\label{propmax} Let $u\in C^{2m}(\overline{B}_1)$ such that
\begin{equation}\label{maxpr}
\left\{
\begin{array}{ll}
(-\Delta)^m u\leq C_1 &\textrm{in } B_1\\
(-\Delta)^j u\leq C_1 & \textrm{on }\partial B_1 \textrm{ for }0\leq j\leq m-1
\end{array}
\right.
\end{equation}
Then there exists a constant $C$ independent of $u$ such that
$$u\leq C\quad \textrm{in }B_1.$$
If $C_1=0$ in \eqref{maxpr}, then $u < 0$ in $B_1$, unless $u\equiv 0$.
\end{prop}

\begin{proof} By induction on $m$. The case $m=1$ follows from the maximum principle, applied to the function $v(x):=u(x)-C|x|^2$, which is subharmonic for $C$ large enough. Assume now that the case $m-1$ has been dealt with and let us consider $u$ satisfying \eqref{maxpr}. Then $v:=-\Delta u$ satisfies $v\leq C$ in $B_1$ by inductive hypothesis. Applying the case $m=1$ again we conclude. Similarly if $C_1=0$.
\end{proof}

\begin{prop}[Fundamental solution]\label{fund} For $m\geq 1$, set
\begin{equation}\label{gammam}
\gamma_{m}:=\omega_{2m} 2^{2m-2}[(m-1)!]^2,
\end{equation}
where $\omega_{2m}:=|S^{2m-1}|=\frac{(2\pi)^m}{(2m-2)!!}$. Then the function
$$K(x):=\frac{1}{\gamma_m}\log\frac{1}{|x|}$$
is a fundamental solution of $(-\Delta)^m$ in $\R{2m}$, i.e. $(-\Delta)^m K=\delta_0$.
\end{prop}

\begin{proof} The case $m=1$ is well-known, so we shall assume $m\geq 2$. Set $r:=|x|$. For radial functions we have
$\Delta=\frac{\partial^2}{\partial r^2}+\frac{n-1}{r}\frac{\partial}{\partial r},$ hence for $j\geq 1$
$$-\Delta \log\frac{1}{r}=\frac{2(m-1)}{r^2}, \qquad -\Delta\frac{1}{r^{2j}}=\frac{4j(m-1-j)}{r^{2j+2}}.\label{eqdelta2}$$
Then
\begin{eqnarray}
\label{eqdelta3}
(-\Delta)^j\log\frac{1}{r}&=& 2^{2j-1}\frac{(j-1)!(m-1)!}{(m-j-1)!}\frac{1}{r^{2j}}\\
\label{Deltalog}
(-\Delta)^{m-1}\log\frac{1}{r}&=&2^{2m-3}(m-2)!(m-1)!\frac{1}{r^{2m-2}}.
\end{eqnarray}
Given a function $\varphi\in C^\infty_c(\R{2m})$, we can apply the usual procedure of integrating by parts in $\R{2m}\bs B_\ve(0)$ using
$$\lim_{\varepsilon\rightarrow 0}\int_{\partial B_\varepsilon(0)}|D^k K|dS=0,\quad 0\leq k\leq 2m-2,$$
to obtain
\begin{eqnarray*}
\int_{\R{2m}}(-\Delta)^m\varphi K dx&=&\lim_{\ve\rightarrow 0}\int_{\partial B_\ve(0)}-\varphi\frac{\partial(-\Delta)^{m-1}K}{\partial \nu}dS\\
&=&\Intm_{\partial B_\ve(0)}\varphi dS\rightarrow \varphi(0).
\end{eqnarray*}
\end{proof}

%\newpage

\end{document}